%% file: Deformation_main.tex
\documentclass[12pt]{amsart}

\input{packages}
\input{macros}
\input{drawings}

\begin{document}

\title{Skein theory and deformations}

\author[Snyder]{Noah~Snyder}
\email{nsnyder1@iu.edu}

\author[Spencer]{Benjamin~Spencer}

\begin{abstract}
    We use skein theoretic `recognition' theorems to classify deformations of certain pivotal or ribbon monoidal categories. We first show that a slight generalization of Kuperberg's characterization of quantum $G_2$ shows that any infinitesimal deformation of quantum $G_2$ as a pivotal category arises by varying $q$. This result applies both at generic $q$, and for the category of tilting modules at $q$ a root of unity (provided we exclude a few small roots of unity).
    Second, we prove a new Kuperberg-like characterization of Deligne's $S_t$ and use it to show that any infinitesimal deformation of $S_t$ (excluding $t=0$) as a ribbon category comes from varying $t$.
\end{abstract}

\maketitle

\section{Introduction} \label{sec:introduction}

The goal of this paper is to study infinitesimal deformations of categories using skein theoretic descriptions of those categories. More specifically, we classify all infinitesimal deformations of two categories: the $G_2$ quantum group category (for $q$ not a primitive $3$rd, $4$th, $6$th, $8$th, $16$th, or $24$th root of unity)\footnote{In fact we expect this argument to work for all $q\neq \pm i$.}, and the category of permutation representations of Deligne's $S_t$ (for $t\neq 0$). These categories have skein theoretic descriptions via planar trivalent graphs (respectively, trivalent ribbon graphs), and our approach is to use classification results of such skein theories.

Giving precise statements of our main results requires some background, so in the introduction we will only give informal statements, with the precise statements in the main paper.

Our starting point is the following theorem of Kuperberg.

\begin{theorem}\cite{MR1265145} \label{thm:intro-G2}
    Suppose that $(\mathcal{C},X)$ is a trivalent category, such that $\dim \Hom(X^{\otimes k}, \mathbf{1}) = 1,0,1,1,4,10$ for $k\leq 5$, such that $\Hom(X^{\otimes k}, \mathbf{1})$ for $0 \leq k \leq 5$ have bases given by the trivalent graphs with no internal faces. Then there is a number $q$ and a functor $(G_2)_q \rightarrow \mathcal{C}$ sending the defining $7$-dimensional representation to $X$.
\end{theorem}

Intuitively, all the conditions in this theorem are `open' and so this should imply the following corollary.

\begin{corollary}\label{cor:intro-G2}
    Any infinitesimal deformation of $(G_2)_q$ comes from varying $q$.\footnote{More precisely, varying $x=q+q^{-1}$.}
\end{corollary}

In order to rigorously state and prove this corollary we will need an improved version of Kuperberg's result which works over an arbitrary commutative base ring $R$ rather than just the complex numbers. However, the proof of this improved statement is identical to that of Kuperberg.

We then prove an analogue of Kuperberg's theorem for the category of permutation representations of Deligne's $S_t$. We introduce a category $\mathcal{S}_t$ which interpolates between the irreducible $(n-1)$-dimensional representations of $S_n$ (i.e. a category version of the quasi-partition algebra \cite{MR3177889} of Daugherty-Orellana), which has a skein theoretic characterization similar to Kuperberg's.

\begin{theorem}
\label{thm:intro-St}
    Suppose that $(\mathcal{C},X)$ is a trivalent ribbon category, such that  $\dim \Hom(X^{\otimes k}, \mathbf{1}) = 1,0,1,1,4$ for $k\leq 4$, and such that these Hom spaces have bases given by the natural basis of partitions with no singletons. Moreover suppose that the $-1$-eigenspace for rotation on $\Hom(X^{\otimes 4},\mathbf{1})$ is one-dimensional. Then there is a number $t$ and a functor $\mathcal{S}_t \rightarrow \mathcal{C}$ sending the non-trivial irreducible summand of the permutation representation to $X$.
\end{theorem}

Again, all these conditions are open, so once we carefully state this theorem over a general base ring we can classify deformations.

\begin{corollary} \label{cor:intro-St}
    Any ribbon infinitesimal deformation of the category of permutation modules for Deligne's $\mathcal{S}_t$ (for $t \neq 0$) comes from varying $t$.
\end{corollary}

Note here that in a ribbon deformation we allow both the monoidal structure and the braiding to vary.

Etingof-Snowden have proved a substantial generalization of Corollary \ref{cor:intro-St} to all Oligomorphic group categories \cite{EtingofSnowdenDeformations}.

We anticipate that similar skein theoretic arguments (inspired by \cite{MR1237835, MR2132671,MR2783128,1507.06030}) can be used to classify all pivotal infinitesimal deformations of quantum group categories in types $A$, $B$, $C$, and $D$, but not in types $E$ and $F$ where we lack skein theoretic descriptions. Although deformations of quantum group categories at generic $q$ have not been classified in the literature, according to Pavel Etingof and ChatGPT, it should be within reach to classify such deformations by computing their Davydov-Yetter cohomology. We do not pursue either of these directions further in the current paper.

In Section \ref{sec:Background} we give some background on categories with a trivalent vertex and deformations. In Section \ref{sec:G2}, we prove our main deformation result for $G_2$. This follows rather easily from a minor modification of \cite{MR1265145}, so in this section we assume that the reader has read that short paper. In Section \ref{sec:StDef}, we define the skein theoretic category $\mathcal{S}_t$ (following \cite[\S6.2.1]{MR5101335}) and show that for $t \neq 0$ it has the same idempotent and additive completion as Deligne's category $\mathrm{Rep}(S_t)$. In Section \ref{sec:StDeformations}, we prove our recognition and deformation theorems for $\mathcal{S}_t$. A reader familiar with trivalent skein theories \cite{MR1265145,MR3624901,MR5101335} should be able to skip directly to Section \ref{sec:StDeformations}.

\subsection*{Acknowledgements}
This work was funded in part by NSF DMS-200009, DMS-2453032, and the Simons Foundation grants MPS-TSM-00007608 and SFI-MPS-SFM-00021230. We thank Pavel Etingof for several helpful conversations.

\subsection*{Authorial and AI statement}
This work is based in part on BS's PhD thesis. This paper was drafted by hand by NS. The main mathematical ideas came from the authors, but ChatGPT 5.6 Sol was used for adapting the paper from analytic deformations to infinitesimal deformations. This adaptation mostly involved routine commutative algebra that NS should have remembered, but ChatGPT did find a key improvement in the argument for Theorem \ref{thm:StMain} allowing us to avoid considering $5$-boxes. AI was also used to draw some TikZ diagrams, for literature search, and to find and fix minor errors. In a few cases (e.g., Lemma \ref{lem:projectiveNakayama}) these suggested edits on the final draft were used verbatim. Many diagrams and a few definitions are reused from \cite{MR5101335}.

\section{Background} \label{sec:Background}

\subsection{Categories and trivalent graphs}

We recall some background on trivalent categories and trivalent ribbon categories from \cite{MR3624901,MR5101335}. We work throughout over a commutative base ring $R$ which we do not assume to be a field. Since $R$ is not assumed to be a field we need to introduce the following notion.

\begin{definition}
    An $R$-linear category is finite free if for every pair of objects $X$, $Y$ we have that $\Hom(X,Y)$ is a finitely generated free $R$-module.
\end{definition}

The reason for this definition is that we will repeatedly use the following version of Nakayama's lemma \cite[Lemma 00DV]{stacks-project}.

\begin{lemma} \label{lem:projectiveNakayama}
    Let $A$ be a commutative ring, let $I\subset A$ be a nilpotent
    ideal, and let $f\colon M\to N$ be a map of finite projective
    $A$-modules. If $f$ is an isomorphism modulo $I$, then $f$ is an
    isomorphism. In particular, if $M/IM$ is free, then any 
    lift of a basis of $M/IM$ is a basis of $M$.
\end{lemma}
\begin{proof}
    Nakayama's lemma implies that $f$ is surjective. Since $N$ is
projective, the resulting short exact sequence splits, so its kernel
$K$ is finite projective. Reducing the split sequence modulo $I$ shows
that $K/IK=0$. Another application of Nakayama's lemma therefore gives
$K=0$.
\end{proof}

We will be studying pivotal finite free $R$-linear categories $\mathcal{C}$ with a distinguished object $X$ with a symmetric self-duality $s\colon X \rightarrow X^*$ (i.e., one where $(p_X)^{-1} \circ s^* = s$ where $p_X \colon X \to X^{**}$ is the pivotal structure). We will also assume that $X$ is a generator in the sense that every object in $\mathcal{C}$ is isomorphic to $X^{\otimes n}$ for some $n$. This is a relatively harmless assumption as you can always restrict to the full subcategory on such objects. Moreover, in the cases we care about, you can recover the original category from this full subcategory by taking the idempotent and additive completions.

Given such a triple $(\mathcal{C}, X, s)$ we let $\mathcal{C}^n$ denote the $n$-box space $\Hom(X^{\otimes n},\mathbf{1})$. Using pivotality, the space of $n$-boxes can be identified with $\Hom(X^{\otimes a},X^{\otimes b})$ for $a+b = n$. Such triples $(\mathcal{C}, X, s)$ are also called unoriented unshaded planar algebras \cite{MR2559686,1507.06030,MR3624901}.

\begin{definition}
    A trivalent (resp. trivalent ribbon) category $(\mathcal{C},X,s,\tau)$ consists of an $R$-linear pivotal (resp. ribbon) category $\mathcal{C}$, a distinguished generator $X$, a symmetric self-duality $s$, and a map $\tau\colon X \otimes X \rightarrow X$ called the trivalent vertex, satisfying:
    \begin{enumerate}
        \item $\tau$ is rotationally invariant. That is, the map $$X \otimes X \rightarrow X \otimes X \otimes X \otimes X^* \rightarrow X^* \otimes X \otimes X \rightarrow X$$ given by \[(\mathrm{ev} \otimes \mathrm{id}) \circ (s \otimes \tau \otimes s^{-1}) \circ (\mathrm{id} \otimes \mathrm{id} \otimes \mathrm{coev})\] is equal to $\tau$.
        \item $\mathcal{C}^k$ is free of rank $1,0,1,1$ for $0 \leq k \leq 3$ with bases $\{\mathrm{id}_{\mathbf{1}}\}$, $\emptyset$, $\{\mathrm{id}_X\}$, $\{\tau\}$.
        \item $s$ and $\tau$ generate all morphisms in the sense that every morphism is a sum of compositions of tensor products of $s$ and $\tau$ together with evaluations, coevaluations, and pivotal isomorphism (resp. and also braidings).
    \end{enumerate}
\end{definition}

Note that in the ribbon case, although the vertex is rotationally invariant, we do not require it to be fixed under the action of the three-strand braid group.

\begin{remark}
    Trivalent categories and trivalent ribbon categories, as defined in  \cite{MR3624901, MR5101335}, included some additional non-degeneracy assumptions. We will not require these non-degeneracy assumptions in this paper, and apologize for any confusion caused by reusing the terminology.
\end{remark}

Let $R$ be a ring and let $\mathsf{PlaTri}_R$ denote the pivotal category of $R$-linear combinations of unoriented planar trivalent graphs. That is, its objects are points on the line, and its morphisms are formal linear combinations of unoriented planar trivalent graphs with composition given by vertical gluing, and tensor product given by horizontal disjoint union. Similarly, let $\mathsf{RibTri}_R$ be the ribbon category of $R$-linear combinations of framed unoriented knotted trivalent graphs. In both cases, we reserve the word `diagram' for the basis vectors.

If $(\mathcal{C}, X, s, \tau)$ is a trivalent category, then the pivotal diagram calculus \cite{MR2767048} defines a full and essentially surjective functor  $\Phi_{(\cC, X, s, \tau)}: \mathsf{PlaTri}_R \rightarrow \cC$ sending the strand to $X$ and the trivalent vertex to $\tau$ (we will often abbreviate this $\Phi_{\cC}$). Similarly, if in addition $\mathcal{C}$ is ribbon, then the ribbon diagram calculus defines a functor  $\mathsf{RibTri}_R \rightarrow \cC$ sending the strand to $X$ and the trivalent vertex to $\tau$. Conversely, a full and essentially surjective pivotal functor $\Phi\colon\mathsf{PlaTri}_R\to\mathcal C$ equips $\Phi(1)$ with the induced symmetric self-duality and trivalent vertex $\Phi(\smallfig{\threevertex})$, and hence defines a trivalent category (and similarly in the ribbon case). Fullness corresponds to the condition that the self-duality and trivalent vertex generate all morphisms.

\begin{definition}
    If $(\mathcal{C},X,s,\tau)$ and $(\mathcal{C}',X',s',\tau')$ are trivalent (resp. trivalent ribbon) categories, then a trivalent (ribbon) functor between them is a (ribbon) functor $\cF: \mathcal{C} \rightarrow \mathcal{C}'$ endowed with an isomorphism $\cF(X) \cong X'$ such that the self-dualities and trivalent vertices satisfy the obvious identities. Equivalently, it is a functor $\cF: \mathcal{C} \rightarrow \mathcal{C}'$ endowed with a monoidal natural isomorphism $\cF \circ \Phi_\cC \cong \Phi_{\cC'}$. We define trivalent natural transformations, and trivalent equivalences similarly.
\end{definition}

Note that this notion of trivalent equivalence is stronger than the usual notion of equivalence.

The categories $\mathsf{PlaTri}_R$ and $\mathsf{RibTri}_R$ have `rescaling' autoequivalences.

\begin{definition} \label{def:rescaling}
    If $u,v \in R^\times$, we have pivotal autoequivalences $F_{u,v}: \mathsf{PlaTri}_R \rightarrow \mathsf{PlaTri}_R$ which send a diagram $D \in \Hom(m,n)$ with $V(D)$ vertices to $u^{\frac{m-n-V(D)}{2}} v^{V(D)} D$ (note that $m-n-V(D)$ is always even, so this formula is well-defined). We also define a ribbon autoequivalence $F_{u,v}: \mathsf{RibTri}_R \rightarrow \mathsf{RibTri}_R$ by the same formulas.
\end{definition}

In particular, $F_{u,v}$ rescales the cap by $u$, the cup by $u^{-1}$, the trivalent vertex in $\Hom(2,1)$ by $v$ and the rotation of the trivalent vertex in $\Hom(1,2)$ by $u^{-1} v$. It rescales the bigon in $\Hom(1,1)$ by $u^{-1} v^2$.
 
\begin{lemma} \label{lem:rescaling}
    If $(\mathcal{C},X,s,\tau)$ is a trivalent category, and $u,v \in R^\times$, then $(\mathcal{C},X,s,\tau)$ is equivalent (but not trivalent equivalent!) to $(\mathcal{C},X,u s,v \tau)$ via $F_{u,v}$.
\end{lemma}

We will often describe trivalent (ribbon) categories by describing them as a quotient of $\mathsf{PlaTri}_R$ (resp. $\mathsf{RibTri}_R$) by certain relations. This means considering the quotient of $\mathsf{PlaTri}_R$ (resp. $\mathsf{RibTri}_R$) by the minimal pivotal (resp. ribbon) ideal generated by the relations. Such a presented category clearly has the universal property that to give a functor out of it is to give a pivotal (resp. ribbon)
category with a distinguished self-duality and rotationally invariant trivalent vertex satisfying all the relations.

\subsection{Deformations}

If $\mathcal{C}$ is an $R$-linear pivotal (resp. ribbon) category and $R \rightarrow S$ is a map of rings, then we can form the base extended category $\mathcal{C} \otimes_R S$ whose objects are the same as those of $\mathcal{C}$ but the morphism spaces are given by $\Hom_{\mathcal{C}}(Y,Z) \otimes_R S.$ We will denote the base extension along $R \rightarrow R/I$ by $\mathcal{C}/I$.

\begin{definition}
    If $(\mathcal{C}, X, s, \tau)$ is a finite free $R$-linear trivalent (resp. trivalent ribbon) category, then an infinitesimal deformation of $\mathcal{C}$ is a finite free $R[h]/h^2$-linear trivalent (resp. trivalent ribbon) category $(\mathcal{C}', X', s', \tau')$ such that the base extension $(\mathcal{C}', X', s', \tau')/h$ is trivalent equivalent to $(\mathcal{C}, X, s, \tau)$. We consider two infinitesimal deformations equivalent if there exist $u,v\in (R[h]/h^2)^\times$ with $u\equiv v\equiv 1 \pmod h$ such that, after replacing the self-duality and trivalent vertex of one deformation by $us'$ and $v\tau'$, respectively, the two deformations are trivalently equivalent.
\end{definition}

Our notion of an infinitesimal deformation of $(\mathcal{C},X,s,\tau)$ agrees with the usual notion of a pivotal infinitesimal deformation of $\mathcal{C}$ in the sense of \cite[Definition~2.1]{MR1641842} (or \cite{MR1664995} in the ribbon case), provided $6\in R^\times$. First we note that the self-duality and trivalent vertex automatically lift to a deformation of $\mathcal{C}$ since the allowed rotational eigenvalues on $\mathcal{C}^2$ and $\mathcal{C}^3$ cannot deform because $6 \in R^\times$. Second, these morphisms still generate the deformation by Nakayama's
    lemma, since every morphism space is finitely generated. It is also not important to distinguish between infinitesimal deformations of $\mathcal{C}$ and infinitesimal deformations of its additive and idempotent completion, because the completion of a deformation is a deformation of the completion by \cite[Proposition~9.1.1]{EGHLSVY11}.

One can similarly define order $n$ deformations and formal $1$-parameter families of deformations by replacing $R[h]/h^2$ with $R[h]/h^{n+1}$ or $R[[h]]$.

\section{Deformations of quantum \texorpdfstring{$G_2$}{G 2}} \label{sec:G2}

\begin{definition}
Let $R$ be a commutative ring. For $x \in R$, let $G_2(R,x)$ be the $R$-linear trivalent category defined by the following relations:

\begin{align*}
\unknot\; &= x^5 + x^4 -5x^3 -4x^2 +6x +3
\end{align*}

\begin{align*}
    \loopvertex\;&=0&
      \twogon\;&= -(x+1)(x^2-2)\;\onestrandid&
        \threegon\; &= x^2-1 \;\threevertex
\end{align*}

\begin{align*}
\ngon[45]{4} &=  -x \left(\;\drawI \; +\; \drawH \; \right) +  (x+1) \left(\; \cupcap \; + \; \twostrandid \; \right) \displaybreak[1] \\[5pt]
\ngon[90]{5} &=  \left(\mathfig{0.1}{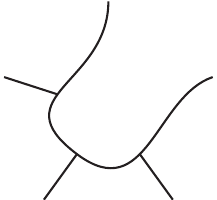} + \mathfig{0.1}{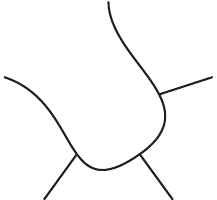} + \mathfig{0.1}{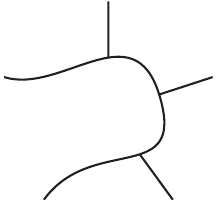} + \mathfig{0.1}{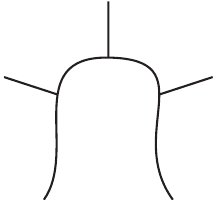} + \mathfig{0.1}{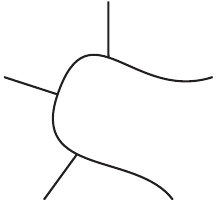} \right) \\
& \qquad -  \left(\mathfig{0.1}{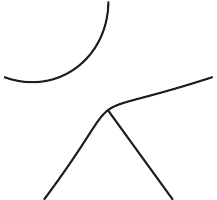} + \mathfig{0.1}{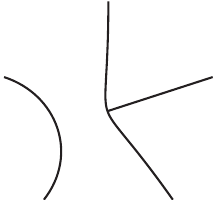} + \mathfig{0.1}{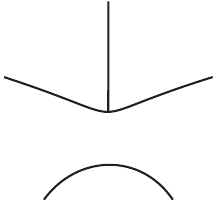} + \mathfig{0.1}{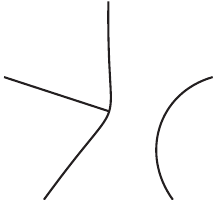} + \mathfig{0.1}{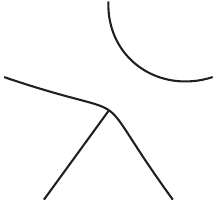}\right) \displaybreak[1]
\end{align*}
\end{definition}

Each of these relations can be thought of as `reducing' relations where you can replace the diagram on the left with the `simpler' diagram on the right. If a diagram does not contain any pentagon or smaller internal faces (so no simplifications are possible) we call the diagrams non-elliptic. Using confluence \cite{MR2308953}, we have the following result.

\begin{proposition}
    Non-elliptic diagrams with $n$-boundary points form an $R$-basis of the $n$-box space of $G_2(R,x)$.
\end{proposition}

\begin{corollary}
    $G_2(R,x)$ is a finite free trivalent category.
\end{corollary}

Note that $G_2(\mathbb{C},q+q^{-1})$ is exactly the definition of Kuperberg's $G_2$ spider. In particular, we have the following two results. 

\begin{theorem}\cite{MR1403861}
    If $q$ is not a root of unity, then the additive idempotent completion of $G_2(\mathbb{C},q+q^{-1})$ is equivalent to the usual category of Type $1$ representations of the quantum group $U_q(\mathfrak{g}_2).$
\end{theorem}

\begin{theorem}\cite{OstrikSnyderG2}
    If $q$ is a root of unity of order not in $\{3,4,6,8,16,24\}$ then the additive and idempotent completion of $G_2(\mathbb{C},q+q^{-1})$ is equivalent to the category of tilting modules of the Lusztig form of the quantum group $U_q(\mathfrak{g}_2).$
\end{theorem}

\begin{remark}
    When $q = \pm i$ the additive and idempotent completion of the Kuperberg spider is \emph{not} equivalent to the category of tilting modules, because the $7$-dimensional representation is not tilting! By contrast we expect that for the other $q$ on the excluded list, the two categories are equivalent, but the argument from \cite{OstrikSnyderG2} must be modified slightly.
\end{remark}

A close reading of Kuperberg's paper \cite{MR1265145} shows that he proves the following generalization of his main result where we work over rings instead of fields.

\begin{theorem} \label{thm:G2recognition}
    Let $(\mathcal{C},X, s, \tau)$ be an $R$-linear trivalent category. Suppose that the $k$-box spaces for $k \leq 5$ are free $R$-modules with bases the planar trivalent graphs with no internal closed faces:
    \begin{multline*}
        \left\{ \right\}; \quad \emptyset; \quad 
        \left\{\; \vertonestrandid \; \right\}; \quad
        \left\{\threevertex\right\}; \;
        \left\{\; \drawH, \drawI, \cupcap, \twostrandid \;\right\}\\
        \left\{\mathfig{0.08}{tree1}, \mathfig{0.08}{tree2}, \mathfig{0.08}{tree3}, \mathfig{0.08}{tree4}, \mathfig{0.08}{tree5},   \mathfig{0.08}{forest1}, \mathfig{0.08}{forest2}, \mathfig{0.08}{forest3}, \mathfig{0.08}{forest4}, \mathfig{0.08}{forest5}\right\}.
    \end{multline*}
Let $a$, $b$, $e_1$, and $e_2$ be the coefficients obtained by rewriting the square and pentagon in these bases:
\begin{align*}
    \ngon[45]{4}
        &= -a\left(\;\drawI+\drawH\right)
        +b\left(\cupcap+\twostrandid \; \right)
        \displaybreak[1] \\[5pt]
    \ngon[90]{5}
        &= e_1\left(
            \mathfig{0.1}{tree1}
            +\mathfig{0.1}{tree2}
            +\mathfig{0.1}{tree3}
            +\mathfig{0.1}{tree4}
            +\mathfig{0.1}{tree5}
        \right) \\
        &\qquad
        +e_2\left(
            \mathfig{0.1}{forest1}
            +\mathfig{0.1}{forest2}
            +\mathfig{0.1}{forest3}
            +\mathfig{0.1}{forest4}
            +\mathfig{0.1}{forest5}
        \right).
\end{align*}
If $e_1$ is invertible in $R$, then there is a trivalent pivotal functor
\[
    G_2(R,e_1^{-1}a)
    \longrightarrow
    (\mathcal{C},X,e_1^{-1}s,e_1^{-1}\tau).
\]
\end{theorem}
\begin{proof}
    Following Kuperberg, we express each small polygon (circle, lollipop, bigon, triangle, square, and pentagon) as a sum of basis vectors. By expanding pairs of adjacent faces in two different ways we get a number of polynomial equations between these variables, which appear on \cite[p. 8]{MR1265145}. Since $e_1 \in R^\times$, we rescale using \ref{lem:rescaling} to assume wlog $e_1 = 1$. Then, following Kuperberg, we can solve for all the other variables.\footnote{As written, Kuperberg divides by $d_2$ when he solves for his variable $a$ using equation (1). Instead we can solve for $a$ using his equation $bc = 2e_1 b + 2e_2 + a e_2 + e_1 c$, since he's already solved $e_2=-1$.} 
\end{proof}

From this ring-theoretic version of Kuperberg's result we can immediately classify infinitesimal deformations of $G_2(\mathbb{C},x)$.

\begin{corollary}
    Any infinitesimal deformation of $G_2(\mathbb{C},x)$ is equivalent to $G_2(\mathbb{C}[h]/h^2,x')$ for $x' \in \mathbb{C}[h]/h^2$ with $x' \equiv x \pmod h$.
\end{corollary}

\begin{proof}
    Let $A=\mathbb{C}[h]/h^2$, and let
    $(\mathcal{C},X,s',\tau')$ be an infinitesimal deformation of
    $G_2(\mathbb{C},x)$. The non-elliptic diagrams form a basis for
    the box spaces of $G_2(\mathbb{C},x)$, and thus, by Lemma \ref{lem:projectiveNakayama},
    their lifts form an $A$-basis for the corresponding box spaces of
    $\mathcal{C}$.

    Let $a'$ and $e_1'$ be the corresponding coefficients for
    $\mathcal{C}$. Since $e_1'\equiv1\pmod h$, we have
    $e_1'\in A^\times$. Set
    \[
        x'=(e_1')^{-1}a'.
    \]
    Since $a'\equiv x\pmod h$ and $e_1'\equiv1\pmod h$, we have
    $x'\equiv x\pmod h$. Thus Theorem~\ref{thm:G2recognition} gives
    a trivalent pivotal functor
    \[
        G_2(A,x')
        \longrightarrow
        (\mathcal{C},X,(e_1')^{-1}s',(e_1')^{-1}\tau').
    \]
    Modulo $h$, this functor is an equivalence, so by Lemma \ref{lem:projectiveNakayama} it is an isomorphism on every box space, and hence it is a trivalent pivotal equivalence. Finally, since
    $(e_1')^{-1}\equiv1\pmod h$, its target is equivalent to
    $(\mathcal{C},X,s',\tau')$ as an infinitesimal deformation.
\end{proof}

\begin{remark}
    The same argument also shows that any order $n$ deformation is given by varying $x$, and taking inverse limits, any formal $1$-parameter family is also given by varying $x$.
\end{remark}

\begin{remark}
    We stated Theorem \ref{thm:G2recognition} for $G_2$ as a pivotal category, although the quantum group categories naturally carry ribbon structures when $q \neq \pm i$. The compatible braidings on the $G_2$ trivalent category were classified in \cite[\S8]{MR3624901}: they are the two standard braidings corresponding to $q$ and $q^{-1}$. The same calculation over $\mathbb{C}[h]/h^2$ shows that every ribbon infinitesimal deformation is obtained by varying $q$. Away from $q=\pm1$, this is equivalent to varying the pivotal parameter $x=q+q^{-1}$. At $q=\pm1$, the map $q\mapsto q+q^{-1}$ has vanishing derivative, so the first-order variation of $q$ is visible in the braiding but not in the underlying pivotal category. That any \emph{analytic} ribbon deformation of $G_2$ for $|q| \neq 1$ comes from varying $q$ should also follow from the main result of \cite{2510.09922}.
\end{remark}

\section{Some variations on Deligne's \texorpdfstring{$S_t$}{S t}} \label{sec:StDef}

Deligne introduced a category which interpolates the defining permutation $n$-dimensional representation of $S_n$. In this section we introduce a small variation on this category, following \cite[\S6.2.1]{MR5101335}, which interpolates between the irreducible $(n-1)$-dimensional representations. This is the category version of the quasi-partition algebras \cite{MR3177889,MR4756485}, whereas Deligne's construction is the category version of the partition algebra of Halverson-Ram \cite{MR2143201}. Because we need to work over a general ring $R$, we repeat some material from \cite{2007.11640,MR5101335}. A reader already familiar with this material can safely skip this section.

\begin{definition}
For $R$ a commutative ring and $a,b \in R$, let $\mathrm{Cob}(R,a,b)$ be the $R$-linear ribbon category generated by a trivalent vertex and a one-valent vertex modulo the following relations:

\begin{align*}
    \unknot\; &= ab  &\qquad
      \twist\; &=  \;\drawcup&\qquad
        \twistvertex\; &=  \;\threevertex
\end{align*}
\begin{align*}
    \loopvertex\;&= b \onevalent&
      \twogon\;&= \;b \; \onestrandid& \dogbone \;&= a &\\
	\leftunitvertex \; &= \; \vertonestrandid \; = \; \rightunitvertex & \braidcross \; &= \; \invbraidcross &\drawH \; &= \; \drawI \; \\
\end{align*}
\end{definition}

Note that the circle and lollipop relations follow from the other relations, as the lollipop is a $1$-valent vertex composed with a bigon and the circle is two $1$-valent vertices composed on each end of the bigon. The name $\mathrm{Cob}$ comes from the viewpoint of \cite{2007.11640}, where this category can also be thought of as the $2$-dimensional bordism category modulo the sphere equals $a$ and the handle morphism being $b$ times the cylinder.

\begin{remark}
    $\mathrm{Cob}(R,t,1)$ is equivalent to the category which Deligne calls $\mathrm{Rep}_1(S_t,R)$ \cite{MR2348906}. This is because both have the same universal property, namely that a functor out of the category is the same data as a separable Frobenius algebra object such that the composition of the unit and counit is multiplication by $t$. The additive and Karoubian completion of $\mathrm{Rep}_1(S_t,R)$ is what Deligne calls $\mathrm{Rep}(S_t,R)$. If $R = \mathbb{C}$ and $t \notin \mathbb{N}$, then $\mathrm{Rep}(S_t,R)$ is a semisimple abelian category, while if $t \in \mathbb{N}$ there is a more subtle construction of an abelian envelope $\mathrm{Rep}(S_t,R)^{\mathrm{ab}}$ \cite{MR2348906, MR2737787,1601.03426}, which we will not need in this paper.
\end{remark}

\begin{lemma}
    If $z \in R^\times$, then $\mathrm{Cob}(R,a, b)$ is ribbon equivalent to $\mathrm{Cob}(R,za, z^{-1}b)$.
\end{lemma}
\begin{proof}
    The equivalence is given by the functor $F_z: \mathrm{Cob}(R,a, b) \rightarrow \mathrm{Cob}(R,za, z^{-1}b)$ which is given on diagrams $D \in \Hom(m,n)$ by $F_z(D) = z^{\frac{n-m+T(D)-P(D)}{2}} D$ where $T(D)$ is the number of trivalent vertices and $P(D)$ is the number of $1$-valent vertices (the quantity $n-m+T(D)-P(D)$ is always even, so this does not require a square root). This functor is well-defined by checking the defining relations, and is an equivalence because the inverse is given by $F_{z^{-1}}$.
\end{proof}

\begin{remark}
In particular, if $ab \neq 0$, we have that $\mathrm{Cob}(\mathbb{C},a,b)$ is also equivalent to Deligne's $\mathrm{Rep}_1(S_{ab},\mathbb{C})$. However, the categories $\mathrm{Cob}(\mathbb{C},1,0)$, $\mathrm{Cob}(\mathbb{C},0,1)$, and $\mathrm{Cob}(\mathbb{C},0,0)$ are not equivalent to each other, and only $\mathrm{Cob}(\mathbb{C},0,1)$ is equivalent to Deligne's $\mathrm{Rep}_1(S_0,\mathbb{C})$. In particular, the moduli space of equivalence classes of $\mathrm{Cob}(\mathbb{C},a,b)$ up to equivalence is a non-Hausdorff space that is a line with the origin tripled. The semisimplification of $\mathrm{Cob}(\mathbb{C},1,0)$ is studied in \cite[\S 5]{MR4450143}.
\end{remark}

\begin{definition}
     For each partition $\pi$ of $\{1, \ldots n\}$, we define a diagram $B_\pi$ in $\mathrm{Cob}(R,a, b)$ with $n$ boundary points as follows. Connect each part of the partition using a left-associated trivalent tree, and have the trees coming from lexicographically earlier parts pass over lexicographically later ones.
 \end{definition}

\begin{example}
    If $\pi$ is the partition
    $\{1,2,4,6\}\bigsqcup\{3,5\}\bigsqcup\{7\}$,
    then the diagram $B_\pi$ is
    \[
    \begin{tikzpicture}[
        baseline=-0.1cm,
        line cap=round,
        line join=round
    ]
        \foreach \i in {1,...,7} {
            \coordinate (p\i) at ({0.8*(\i-1)},0);
            \node[below=2pt] at (p\i) {\scriptsize $\i$};
        }

        \draw[line width=0.7pt]
            (p3) .. controls (1.6,0.8) and (3.2,0.8) .. (p5);

        \coordinate (vone) at (0.4,0.5);
        \coordinate (vtwo) at (2.1,1.1);

        \draw[white,line width=3pt]
            (p1) -- (vone)
            (p2) -- (vone)
            (vone) -- (vtwo)
            (p4) -- (vtwo)
            (p6) -- (vtwo);

        \draw[black,line width=0.7pt]
            (p1) -- (vone)
            (p2) -- (vone)
            (vone) -- (vtwo)
            (p4) -- (vtwo)
            (p6) -- (vtwo);

        \coordinate (vthree) at (4.8,0.6);
        \draw[line width=0.7pt] (p7) -- (vthree);
        \fill (vthree) circle (1.3pt);
    \end{tikzpicture}.
    \]
\end{example}

The following lemma is immediate from the cobordism viewpoint.
  
\begin{lemma}
     The $\{B_\pi\}$ are an $R$-basis for $\mathrm{Cob}(R,a,b)$.
 \end{lemma}

\begin{remark}
    The basis that Deligne uses for $\mathrm{Rep}_1(S_{ab},R)$ is also indexed by partitions, but is \emph{not} the basis corresponding to $\{B_\pi\}$. This is most naturally seen from the Harman-Snowden description \cite{2204.04526}, where the diagrammatic basis consists of the Schwartz functions on $\mathbb{N}^n$ which are $1$ on exactly the tuples where $x_i = x_j$ when $i$ and $j$ are in the same part (with no condition at all on distinct parts), while Deligne's basis is the functions which are $1$ on exactly the tuples where $x_i = x_j$ if $i$ and $j$ are in the same part and $x_i \neq x_j$ if they are in different parts. The change of basis matrix between these two bases is triangular with $1$'s on the diagonal with respect to the refinement partial ordering.
\end{remark}

If $a \in R^\times$, then we have an idempotent
\[P = \vertonestrandid \;- \frac{1}{a} \counitunit.\]
This gives an object $X$ in the idempotent completion $\mathrm{Kar}(\mathrm{Cob}(R,a,b))$. The strand in $\mathrm{Cob}(R,a,b)$ is equivalent to $\mathbf{1} \oplus X$.

\begin{definition}
    If $a \in R^\times$, let $\mathrm{Cob}'(R,a,b)$ denote the full subcategory of $\mathrm{Kar}(\mathrm{Cob}(R,a,b))$ whose objects are the tensor powers of $X$.
\end{definition}

Any diagram in $\mathrm{Cob}(R,a,b)$ can be restricted to $\mathrm{Cob}'(R,a,b)$ by pre and post composing with $P$, and we use the notation $\mathrm{Res}(D)$ for restriction. However be warned that \emph{restriction is not functorial}, because the composition of restrictions inserts a $P$ on every edge while the restriction of the composition only inserts them on boundary edges.

We have that $\mathrm{Cob}'(R,a,b)$ has a self-duality given by the identity, and a trivalent vertex given by the restriction of the trivalent vertex. We aim to give a direct diagram description of $\mathrm{Cob}'(R,a,b)$ in terms of this trivalent vertex.

\begin{definition}
If $a \in R^\times$, let $\mathcal{S}(R,a,b)$ be the $R$-linear ribbon category given by the following relations.
\begin{align*}
    \unknot\; &= ab-1 &\qquad
      \twist\; &=  \;\drawcup&\qquad
        \twistvertex\; &=  \;\threevertex\\[5pt] 
    \loopvertex\;&=0&
      \twogon\;&= \frac{ab-2}{a} \;\onestrandid & \braidcross &= \invbraidcross
\end{align*}
\vspace{5pt}
\[ \qquad  \drawH \; - \; \drawI \; +  \; \frac{1}{a}\left(\;\twostrandid \; -  \; \cupcap \; \right) = 0.\]
\end{definition}

\begin{lemma} \label{lem:bigoncheck}
    There's a functor $\mathcal{S}(R,a,b) \rightarrow \mathrm{Cob}'(R,a,b)$.
\end{lemma}
\begin{proof}
    We need to check that all the defining identities hold in $\mathrm{Cob}'(R,a,b)$ by interpreting them in $\mathrm{Cob}(R,a,b)$ by inserting a projection $P$ at every edge. We check the bigon relation, the others are similar. 

\[\PBigon{1}{1}
=
\PBigon{0}{0}
-\frac{1}{a} \; \PBigon{2}{0}
-\frac{1}{a} \; \PBigon{0}{2}
+\frac{1}{a^2} \; \PBigon{2}{2}
= \left(b-\frac{1}{a} -\frac{1}{a} + 0 \right) \; \PIdentity =
\frac{ab-2}{a} \; \PIdentity.\]

\end{proof}

\begin{proposition} \label{prop:StVersions}
The functor $\mathcal{S}(R,a,b) \rightarrow \mathrm{Cob}'(R,a,b)$ is an equivalence.
\end{proposition}

We show this by exhibiting a spanning set for $\mathcal{S}(R,a,b)$ whose image in $\mathrm{Cob}'(R,a,b)$ is a basis.

\begin{definition}
    For each partition $\pi$ of $\{1, \ldots n\}$ with no singletons, we define a diagram $D_\pi$ in $\mathcal{S}(R,a,b)$ with $n$ boundary points as follows. Connect each part of the partition using a left-associated trivalent tree, and have the trees coming from lexicographically earlier parts pass over lexicographically later ones. We let $D'_\pi$ denote the image of $D_\pi$ in $\mathrm{Cob}'(R,a,b)$.
\end{definition}

\begin{lemma}
     The diagrams $\{D_\pi\}$ where $\pi$ is a partition with no singletons span $\mathcal{S}(R,a,b)$.
\end{lemma}
\begin{proof}
    By induction on the number of faces and the size of the smallest face we see that $\mathcal{S}(R,a,b)$ is spanned by forests. Using the last relation, and induction on the number of vertices, it is spanned by left-associated forests. Using the symmetry of the crossing, we may then arrange that the tree associated to a lexicographically earlier part passes over the trees associated to later parts.
\end{proof}

\begin{lemma}
    The restrictions $\mathrm{Res}(B_\pi)$ for $\pi$ ranging over partitions with no singletons form a basis of $\mathrm{Cob}'(R,a,b)$.
\end{lemma}
\begin{proof}
    Since the $\{B_\pi\}$ for $\pi$ a partition (perhaps with singletons) are a basis of $\mathrm{Cob}(R,a,b)$, their restrictions to $\mathrm{Cob}'(R,a,b)$ span $\mathrm{Cob}'(R,a,b)$. The restriction of a partition with a singleton is zero, so the $\{\mathrm{Res}(B_\pi)\}$ for $\pi$ a partition with no singletons still spans $\mathrm{Cob}'(R,a,b)$. 
    
    We now show that the $\{\mathrm{Res}(B_\pi)\}$ for $\pi$ a partition with no singletons are linearly independent. Expanding each $P$ on the boundary
     \[
    \mathrm{Res}(B_\pi)
      = B_\pi +
        \sum_{\sigma}
        c_{\pi,\sigma} B_\sigma,
    \]
    where the sum is taken over partitions with at least one singleton. Ordering the partition basis so that the singleton-free partitions come first, these expansions are in echelon form with pivots $B_\pi$. Thus the elements $\mathrm{Res}(B_\pi)$ are linearly independent.
\end{proof}

\begin{lemma} \label{lem:reducedpartitionbasis}
     $\{D'_\pi\}$ where $\pi$ ranges over partitions with no singletons is a basis of $\mathrm{Cob}'(R,a,b)$.
\end{lemma}
\begin{proof}
    The $D'_\pi$ are given by inserting a $P$ at every edge including internal edges, so expanding every internal edge, while the $\mathrm{Res}(B_\pi)$ are given by inserting $P$ only at external edges. Expanding the definition of $P$ at every internal edge, we see that 
    \[
    D'_\pi
      = \mathrm{Res}(B_\pi) +
        \sum_{\sigma < \pi}
        c_{\pi,\sigma}\mathrm{Res}(B_\sigma),
    \]
    where the sum is taken over partitions $\sigma$ with no singletons which are strict refinements of $\pi$, thus the $\{D'_\pi\}$ form a basis.
\end{proof}

Since the spanning set $\{D_\pi\}$ of $\mathcal{S}(R,a,b)$ has image $\{D'_\pi\}$ in $\mathrm{Cob}'(R,a,b)$ which is a basis, Proposition \ref{prop:StVersions} is proven.

\begin{corollary}
    $\mathcal{S}(R,a,b) \cong \mathrm{Cob}'(R,a,b)$ is a finite free trivalent ribbon category.
\end{corollary}
\begin{proof}
    This follows immediately from Lemma \ref{lem:reducedpartitionbasis}.
\end{proof}

%


\section{Ribbon deformations of \texorpdfstring{$\mathcal{S}_t$}{S t}} \label{sec:StDeformations}

We prove a Kuperberg-like characterization of the trivalent version of Deligne's $S_t$, which we call $\mathcal{S}_t(R)$. The key new ingredient is to add a rotational eigenspace assumption.

\begin{definition}
For $t \in R$, let $\mathcal{S}_t(R)$ be the $R$-linear ribbon category given by the following relations.
\begin{align*}
    \unknot\; &= t-1 &\qquad
      \twist\; &=  \;\drawcup&\qquad
        \twistvertex\; &=  \;\threevertex\\[5pt] 
    \loopvertex\;&=0&
      \twogon\;&= (t-2) \;\onestrandid \; & \braidcross &= \invbraidcross 
\end{align*}
\vspace{5pt}
\[\drawH \; - \; \drawI \; +  \;\twostrandid \; -  \; \cupcap = 0.\]
\end{definition}

This is precisely the category $\mathcal{S}(R,1,t)$ from the last section. We will need from the last section the following facts:

\begin{itemize}
    \item $\mathcal{S}_t(R)$ is a finite free trivalent ribbon category.
    \item An $R$-basis for $\mathcal{S}_t(R)^k$ for $k \leq 4$ is given by the diagrams:
    \begin{equation*}
        \left\{ \right\}; \quad \emptyset; \quad 
        \left\{\; \vertonestrandid \; \right\}; \quad
        \left\{\threevertex\right\}; \;
        \left\{\;   \twostrandid,\; \cupcap,\; \drawI,\; \braidcross \;\right\}
    \end{equation*}
    \item For $t \neq 0$ the additive and idempotent completion of $\mathcal{S}_t(\mathbb{C})$ is equivalent to Deligne's $\mathrm{Rep}(S_t)$.
\end{itemize}

\begin{theorem} \label{thm:StMain}
    Let $(\mathcal{C},X, s, \tau)$ be a trivalent ribbon category over a commutative ring $R$ with $2 \in R^\times$. Suppose that the $k$-box spaces $\mathcal{C}^k$ for $k \leq 4$ are free $R$-modules of ranks $1,0,1,1,4$ with bases given by the diagrams:

    \begin{equation*}
        \left\{ \right\}; \quad \emptyset; \quad 
        \left\{\; \vertonestrandid \; \right\}; \quad
        \left\{\threevertex\right\}; \;
        \left\{\;   \twostrandid,\; \cupcap,\; \drawI,\; \braidcross \;\right\}
    \end{equation*}

    Further suppose that the $-1$-eigenspace for rotation on $\mathcal{C}^4$ is free of rank $1$.

    Define $t$ so that the circle is $t-1$ times the empty diagram. Let $y_i$ be the coefficients in
    \[\drawH =  y_1 \twostrandid + y_2 \cupcap + y_3 \drawI + y_4 \braidcross.\]

     If $y_1$ is invertible, then there's a trivalent ribbon functor $\mathcal{S}_t(R) \rightarrow (\mathcal{C},X,-y_1^{-1}s,-y_1^{-1}\tau)$.
\end{theorem}
\begin{proof}
    We need to check that (perhaps after rescaling) all the defining relations of $\mathcal{S}_t(R)$ hold.

    The relation $\smallfig{\loopvertex} = 0$ is immediate because $\mathcal{C}^1$ is the zero vector space.

    From the basis assumption, we have $t,\alpha,\beta \in R$ such that 
    \[\unknot\; = t-1 \qquad
      \twist\; =  \alpha \;\drawcup \qquad
        \twistvertex\; =  \beta \;\threevertex.\]

    Since $2 \in R^\times$ and the $-1$ rotational eigenspace is free of rank $1$, we have relations:
    \begin{equation} \label{eq:IequalsH}
        \drawI - \drawH = r \left(\; \twostrandid -\cupcap \; \right)
    \end{equation}
    \begin{equation} \label{eq:crossing}
    \braidcross - \invbraidcross = \rho \left(\; \twostrandid -\cupcap \; \right)    
    \end{equation}
    
    Thus $y_1 = -r$. Since $y_1$ is invertible, so is $r$. Replacing $s$ and $\tau$ by $r^{-1}s$ and $r^{-1}\tau$ (as in Definition \ref{def:rescaling}), wlog $r = 1$.

    Pulling the upper-right vertex through the crossing gives us the relation
    \begin{equation}
         \twistone = \beta \hspace{0.1in} \twisttwo.
    \end{equation}
    
   We apply Equation \ref{eq:IequalsH} to both sides to get:
   \begin{align*}
       \beta \drawI - \braidcross + \alpha \cupcap &= \beta^2 \drawH  + \beta \alpha \twostrandid - \beta \braidcross \\ &= \beta^2 \drawI + (\beta \alpha-\beta^2) \twostrandid + \beta^2 \cupcap - \beta \braidcross
   \end{align*}
   Comparing coefficients of $\smallfig{\braidcross}$ we see $\beta = 1$, then comparing coefficients of $\smallfig{\cupcap}$ we see $\alpha = 1$.

   The bigon relation follows from applying a cap to Equation \ref{eq:IequalsH}.

   Now we return our attention to Equation \ref{eq:crossing}. Composing with a trivalent vertex and using $\alpha = \beta = 1$, we see that $0 = \rho$.

   Thus all the relations are checked.
\end{proof}

\begin{remark}
Note that the assumption on the $-1$-eigenspace is essential. First, quantum $G_2$ satisfies the assumptions of the theorem, but has $2$-dimensional $-1$-eigenspace. A second example (again with $2$-dimensional $-1$-eigenspace) comes from \cite{SpencerThesis}: the Deligne tensor product $\mathrm{Fib} \boxtimes \mathrm{Fib}^{\mathrm{rev}}$ of a Fibonacci category and its own reversal (i.e. with the same circle value but the inverse braiding).    
\end{remark}

\begin{corollary} \label{cor:StMain}
    Any infinitesimal deformation of $\mathcal{S}_t(\mathbb{C})$ is equivalent to $\mathcal{S}_{t'}(\mathbb{C}[h]/h^2)$ for $t' \in \mathbb{C}[h]/h^2$ with $t' \equiv t \pmod h$.
\end{corollary}
\begin{proof}
    Let $A=\mathbb{C}[h]/h^2$, and let $(\mathcal{C},X,s,\tau)$ be an infinitesimal deformation of $\mathcal{S}_t(\mathbb{C})$.

    By Lemma \ref{lem:projectiveNakayama}, the diagrams above form a $\mathbb{C}[h]/h^2$-basis for $\mathcal{C}^k$ for $k\leq 4$. Let $t'-1$ be the circle value and $y_i'$ be the coefficients in Theorem \ref{thm:StMain} for $\mathcal{C}$. Since $y_1' \equiv -1 \pmod{h}$, we have that $y_1' \in (\mathbb{C}[h]/h^2)^\times$. We now check the eigenspace condition. Let $\rho$ denote rotation on $\mathcal{C}^4$. Since $\rho^2=1$ and $2$ is invertible in $A$, the operator $e_{-1}=\frac{1}{2}(1-\rho)$ is an idempotent whose image is the $(-1)$-eigenspace of $\rho$. Thus $\operatorname{im}(e_{-1})$ is a finite projective $A$-module which is free of rank $1$ modulo $h$, so Lemma~2.2 implies that $\operatorname{im}(e_{-1})$ is free of rank one.
   
    Thus Theorem~\ref{thm:StMain} gives a functor
\[
\mathcal{S}_{t'}(\mathbb{C}[h]/h^2)
\longrightarrow
(\mathcal{C},X,-(y_1')^{-1}s,-(y_1')^{-1}\tau).
\]
Modulo $h$, this functor is an equivalence, so by Lemma~\ref{lem:projectiveNakayama} it is an isomorphism on every box space, and hence is a trivalent ribbon equivalence. Since $-(y_1')^{-1}\equiv 1\pmod h$, the target is equivalent to
$(\mathcal{C},X,s,\tau)$ as an infinitesimal deformation.
\end{proof}

\begin{remark}
    The same argument also shows that any order $n$ deformation is given by varying $t$, and taking inverse limits, any formal $1$-parameter family is also given by varying $t$.
\end{remark}

\bibliographystyle{amsalpha}
\bibliography{deformations}

\end{document}

%% file: packages.tex
\newif\ifblinded
\blindedtrue
\usepackage{amssymb}
\usepackage{array}
\usepackage{booktabs}
\usepackage{mdwtab}
\usepackage[T1]{fontenc}
\usepackage[utf8]{inputenc}
\usepackage{libertine}
\usepackage[libertine]{newtxmath}
\usepackage[scaled=0.96]{zi4}
\usepackage{hyphenat}
\usepackage{enumitem}
\usepackage{xcolor}
\definecolor{medium-blue}{rgb}{0,0,0.65}
\usepackage[pdftex]{graphicx}
\usepackage[pdftex,margin=1in]{geometry}
\usepackage{url}
\usepackage[textsize=small]{todonotes}
\usepackage{makecell} 
\usepackage[tableposition=below]{caption}

\usepackage{silence}
\usepackage{tikz}
\usetikzlibrary{shapes}
\usetikzlibrary{calc}
\usetikzlibrary{knots}

\usepackage{hyperref}

%% file: macros.tex
\DeclareMathOperator{\Hom}{Hom}

\newcommand{\arXiv}[1]{\href{http://arxiv.org/abs/#1}{\tt \nolinkurl{arXiv:#1}}}
\newcommand{\doi}[1]{\href{http://dx.doi.org/#1}{{\tt \nolinkurl{DOI:#1}}}}

\theoremstyle{plain}
\newtheorem{theorem}{Theorem}
\newtheorem{proposition}{Proposition}
\numberwithin{proposition}{section}
\newtheorem{lemma}[proposition]{Lemma}
\newtheorem{corollary}[proposition]{Corollary}

\theoremstyle{definition}
\newtheorem{definition}[proposition]{Definition}

\theoremstyle{remark}
\newtheorem{example}[proposition]{Example}
\newtheorem{remark}[proposition]{Remark}

\makeatletter
\newcommand\mi@kern[1]{%
  \settowidth\@tempdima{$\mi@obj^{#1}$}
  \kern-\@tempdima
  #1
  \settowidth\@tempdima{$\mi@obj$}
  \kern\@tempdima
}

\newtoks\mi@toksp
\newtoks\mi@toksb
\DeclareRobustCommand{\manyindices}[5]{
  \def\mi@obj{#5}
  \mi@toksp\expandafter{\mi@kern{#2}}
  \mi@toksb\expandafter{\mi@kern{#1}}
  \@mathmeasure4\textstyle{#5_{#1}^{#2}}
  \@mathmeasure6\textstyle{#5_{#3}^{#4}}
  \dimen0-\wd6 \advance\dimen0\wd4
  \@mathmeasure8\textstyle{\hphantom{{}_{#1}^{#2}}#5^{\the\mi@toksp#4}_{\the\mi@toksb#3}}
  \hbox to \dimen0{}{\kern-\dimen0\box8}
}
\makeatother 

\def\semicolon{;}
\def\applytolist#1{
    \expandafter\def\csname multi#1\endcsname##1{
        \def\multiack{##1}\ifx\multiack\semicolon
            \def\next{\relax}
        \else
            \csname #1\endcsname{##1}
            \def\next{\csname multi#1\endcsname}
        \fi
        \next}
    \csname multi#1\endcsname}

\def\calc#1{\expandafter\def\csname c#1\endcsname{{\mathcal #1}}}
\applytolist{calc}QWERTYUIOPLKJHGFDSAZXCVBNM;
\def\bbc#1{\expandafter\def\csname bb#1\endcsname{{\mathbb #1}}}
\applytolist{bbc}QWERTYUIOPLKJHGFDSAZXCVBNM;
\def\bfc#1{\expandafter\def\csname bf#1\endcsname{{\mathbf #1}}}
\applytolist{bfc}QWERTYUIOPLKJHGFDSAZXCVBNM;

\DeclareMathAlphabet{\mathbbold}{U}{bbold}{m}{n}

\graphicspath{{diagrams/}}

\newcommand{\mathfig}[2]{{\hspace{-3pt}\begin{array}{c}%
  \raisebox{-2.5pt}{\includegraphics[width=#1\textwidth]{#2}}%
\end{array}\hspace{-3pt}}}

\newread\testin

\def\mathcenter#1{%
  \vcenter{\hbox{$#1$}}%
}



%% file: drawings.tex
\usepackage{fp}
\usepackage{tikz}
\usetikzlibrary{matrix}
\usetikzlibrary{arrows,backgrounds,patterns,scopes,external,hobby,
    decorations.pathreplacing,
    decorations.pathmorphing
}
\usepackage{tikzit}

\newlength{\fuzzwidth}
\newlength{\arrowlength}
\newlength{\arrowwidth}
\newlength{\pointrad}
\newlength{\linewid}
\newlength{\circlerad}
\newlength{\smcirclerad}
\newcommand{\fuzzcolor}{black!25}
\newcommand{\arrowcolor}{black!25}
\newcommand{\covercolor}{black!0}
\newcommand{\graycolor}{black!55}
\newcommand{\graylightcolor}{black!40}

\newcommand{\coverwidthfuzz}{6pt}
\newcommand{\coverwidth}{3.5pt}
\newcommand{\coverwidththin}{3.25pt}
\newcommand{\coverwidththick}{3.75pt}

\newlength{\linewidthin}
\newlength{\linewidthick}
\tikzset{cdlabel/.style={execute at begin node=$\scriptstyle,execute at end node=$}}

\tikzset{use Hobby shortcut}
\tikzset{
	coverline/.style={
	preaction={draw,line width=\coverwidth,\covercolor}}, 
	coverlinethin/.style={
	preaction={draw,line width=\coverwidththin,\covercolor}}, 
	coverlinethick/.style={
	preaction={draw,line width=\coverwidththick,\covercolor}}, 
	coverlineleft/.style={
	preaction={draw,line width=\coverwidthfuzz,\covercolor,decorate,decoration={curveto,amplitude=0,raise=.35*\fuzzwidth}}}, 
coverlinelefttail/.style={
	preaction={draw,line width=\coverwidthfuzz,\covercolor,decorate,decoration={curveto,amplitude=0,raise=.35*\fuzzwidth,pre=moveto,pre length=2pt}}}, 
        fuzzlefttail/.style={
        preaction={draw,line width=\fuzzwidth,\fuzzcolor,decorate,decoration={curveto,pre=moveto,pre length=2pt,amplitude=0,raise=.5*\fuzzwidth}}}, 
        linestylethin/.style={line width=\linewidthin},
        linestylethick/.style={line width=\linewidthick},
        linestylegray/.style={line width=\linewid,\graycolor},
        linestylegraylight/.style={line width=\linewid,\graylightcolor}
}
\tikzset{
        fuzzright/.style={
        preaction={draw,line width=\fuzzwidth,\fuzzcolor,decorate,decoration={curveto,amplitude=0,raise=-.5*\fuzzwidth}}},
        fuzzleft/.style={
        preaction={draw,line width=\fuzzwidth,\fuzzcolor,decorate,decoration={curveto,amplitude=0,raise=.5*\fuzzwidth}}},
        fuzzrightpre/.style={ 
        preaction={draw,line width=2pt,\fuzzcolor,decorate,decoration={curveto,amplitude=0,raise=-1pt,pre=moveto,pre length=12pt}}},
        fuzzleftpre/.style={ 
        preaction={draw,line width=2pt,\fuzzcolor,decorate,decoration={curveto,post=moveto,post length=32pt,amplitude=0,raise=1pt}}},        
        outstyle/.style={\arrowcolor, line width=\arrowwidth},
        linestyle/.style={line width=\linewid}
}

\newcommand{\newfig}[3]{
  \newsavebox{#2}
  \savebox{#2}{#3}
  \newcommand{#1}{\usebox{#2}}
}

\newcommand{\smallfig}[1]{\mathcenter{\scalebox{0.5}{#1}}}

\newfig{\twistedsquarehor}{\twistedsquarehorbox}{
\begin{tikzpicture}[baseline=-.5ex,scale=.8]
\draw (45:.8cm) -- (-135:.8cm);
\draw[line width=1mm,white,double=black] (-45:.8cm) -- (135:.8cm);
\draw (45:.5cm) .. controls +(-75:.3cm) and +(75:.3cm) .. (-45:.5cm);
\draw (135:.5cm) .. controls +(-105:.3cm) and +(105:.3cm) .. (-135:.5cm);
\end{tikzpicture}}

\newfig{\twistedsquarever}{\twistedsquareverbox}{
\begin{tikzpicture}[baseline=-.5ex,scale=.8]
\draw (45:.8cm) -- (-135:.8cm);
\draw[line width=1mm,white,double=black] (-45:.8cm) -- (135:.8cm);
\draw (45:.5cm) .. controls +(165:.3cm) and +(15:.3cm) .. (135:.5cm);
\draw (-45:.5cm) .. controls +(-165:.3cm) and +(-15:.3cm) .. (-135:.5cm);
\end{tikzpicture}}

\newcommand{\ngon}[2][0]{
\begin{tikzpicture}[baseline=-0.5ex,scale=0.8]
\foreach \x in {1, ..., #2}
	\draw (360*\x/#2+#1:.8cm)--(360*\x/#2+#1:.5cm);
\foreach \x in {1, ..., #2}
	\draw (360*\x/#2+#1:.5cm) .. controls +(360*\x/#2+120+#1:.3cm) and +(360*\x/#2+360/#2-120+#1:.3cm) .. (360*\x/#2+360/#2+#1:.5cm);
\end{tikzpicture}
}

\newcommand{\nvertex}[2][0]{
\begin{tikzpicture}[baseline=-0.5ex,scale=0.8]
\foreach \x in {1, ..., #2}
	\draw (360*\x/#2+#1:.8cm)--(0,0);
\end{tikzpicture}
}

\newfig{\twogon}{\twogonbox}{\ngon{2}}
\newfig{\threegon}{\threegonbox}{\ngon[-90]{3}}
\newfig{\threevertex}{\threevertexbox}{\nvertex[-90]{3}}
\newfig{\fourgon}{\fourgonbox}{\ngon[45]{4}}

\newfig{\upsidedownthreevertex}{\upsidedownthreevertexbox}{%
\begin{tikzpicture}[baseline=-0.5ex,scale=0.8]
	\draw (330:.8cm)--(0,0);
	\draw (90:.8cm)--(0,0);
	\draw (210:.8cm)--(0,0);
\end{tikzpicture}}

\newfig{\unknot}{\unknotbox}{%
\begin{tikzpicture}[baseline=-0.5ex,scale=0.8]
  \draw (0,0) circle (.6cm);
\end{tikzpicture}}

\newfig{\drawI}{\drawIbox}{\begin{tikzpicture}[baseline=-0.5ex,scale=0.8]
 	\draw (0,.2) -- (45:.8cm);
 	\draw (0,.2) -- (135:.8cm);
	\draw (0,.2) -- (0,-.2);
 	\draw (0,-.2) -- (-45:.8cm);
 	\draw (0,-.2) -- (-135:.8cm);
\end{tikzpicture}}

\newfig{\drawH}{\drawHbox}{\begin{tikzpicture}[baseline=-0.5ex,rotate=90,scale=0.8]
 	\draw (0,.2) -- (45:.8cm);
 	\draw (0,.2) -- (135:.8cm);
	\draw (0,.2) -- (0,-.2);
 	\draw (0,-.2) -- (-45:.8cm);
 	\draw (0,-.2) -- (-135:.8cm);
\end{tikzpicture}}

\newfig{\drawdottedH}{\drawdottedHbox}{\begin{tikzpicture}[baseline=-0.5ex,scale=0.8]
        \draw (45:.8cm) -- (-45:.8cm);
        \draw (135:.8cm) -- (-135:.8cm);
        \draw[dotted] (-0.57cm,0) -- (0.57cm, 0);
\end{tikzpicture}}

\newfig{\onestrandid}{\onestrandidbox}{\begin{tikzpicture}[baseline=-0.5ex,scale=0.8]
	\draw (-.8cm,0)--(.8cm,0);
\end{tikzpicture}}

\newfig{\vertonestrandid}{\vertonestrandidbox}{\begin{tikzpicture}[rotate=90,baseline=-0.5ex,scale=0.8]
	\draw (-.8cm,0)--(.8cm,0);
\end{tikzpicture}}

\newfig{\drawcup}{\drawcupbox}{\begin{tikzpicture}[baseline=-0.5ex,scale=0.8]
	\draw (45:.8cm) to [curve through=(90:0cm)] (135:.8cm);
\end{tikzpicture}}

\newfig{\drawcap}{\drawcapbox}{\begin{tikzpicture}[baseline=-0.5ex,scale=0.8]
	\draw (-45:.8cm) to [curve through=(-90:0cm)] (-135:.8cm);
\end{tikzpicture}}

\newfig{\drawbubblecap}{\drawbubblecapbox}{\begin{tikzpicture}[baseline=-0.5ex,scale=0.8]
	\draw (180:.8cm) to [curve through = (135: .8cm)] (120:.8cm);
	\draw (0:.8cm) to [curve through = (45: .8cm)] (60:.8cm);
	\draw (120:.8cm) to [curve through = (90:1cm)] (60:.8cm);
	\draw (120:.8cm) to [curve through = (90:.4cm)] (60:.8cm);
\end{tikzpicture}}

\newfig{\cupcap}{\cupcapbox}{\begin{tikzpicture}[baseline=-0.5ex,scale=0.8]
	\draw (45:.8cm) to [curve through=(90:.3cm)] (135:.8cm);
	\draw (-45:.8cm) to [curve through=(-90:.3cm)] (-135:.8cm);
\end{tikzpicture}}

\newfig{\twostrandid}{\twostrandidbox}{\begin{tikzpicture}[baseline=-0.5ex,rotate=90,scale=0.8]
	\draw (45:.8cm) to [curve through=(90:.3cm)] (135:.8cm);
	\draw (-45:.8cm) to [curve through=(-90:.3cm)] (-135:.8cm);
\end{tikzpicture}}

\newfig{\symcross}{\symcrossbox}{\begin{tikzpicture}[baseline=-0.5ex,scale=0.8]
	\draw (45:.8cm) -- (-135:.8cm);
	\draw (-45:.8cm) -- (135:.8cm);
\end{tikzpicture}}

\newfig{\braidcross}{\braidcrossbox}{\begin{tikzpicture}[baseline=-0.5ex,scale=0.8]
	\draw (45:.8cm) -- (-135:.8cm);
	\draw[line width=1mm,white,double=black] (-45:.8cm) -- (135:.8cm);
\end{tikzpicture}}

\newfig{\invbraidcross}{\invbraidcrossbox}{\begin{tikzpicture}[baseline=-0.5ex,scale=0.8]
	\draw (-45:.8cm) -- (135:.8cm);
	\draw[line width=1mm,white,double=black] (45:.8cm) -- (-135:.8cm);
\end{tikzpicture}}

\newfig{\drawcrossX}{\drawcrossXbox}{\begin{tikzpicture}[baseline=-0.5ex,scale=0.8,rotate=90]
        \draw ([out angle=-150].2,-.3) ..([in angle=-70]135:.8cm);
 	\draw[line width=2mm,white]
              ([out angle=150].2,.3) .. ([in angle=70]-135:.8cm);
 	\draw ([out angle=150].2,.3) .. ([in angle=70]-135:.8cm);
 	\draw ([out angle=30].2,.3) .. (45:.8cm);
	\draw (.2,.3) -- (.2,-.3);
 	\draw ([out angle=-30].2,-.3) .. (-45:.8cm);
\end{tikzpicture}}

\newfig{\drawinvcrossX}{\drawinvcrossXbox}{\begin{tikzpicture}[baseline=-0.5ex,scale=0.8,rotate=90]
 	\draw ([out angle=150].2,.3) .. ([in angle=70]-135:.8cm);
        \draw[line width=2mm,white] ([out angle=-150].2,-.3) ..([in angle=-70]135:.8cm);
        \draw([out angle=-150].2,-.3) ..([in angle=-70]135:.8cm);
	\draw ([out angle=30].2,.3) .. (45:.8cm);
	\draw (.2,.3) -- (.2,-.3);
 	\draw ([out angle=-30].2,-.3) .. (-45:.8cm);
\end{tikzpicture}}

\newfig{\drawsymcrossX}{\drawsymcrossXbox}{\begin{tikzpicture}[baseline=-0.5ex,scale=0.8,rotate=90]
 	\draw ([out angle=30].2,.3) .. (45:.8cm);
	\draw (.2,.3) -- (.2,-.3);
 	\draw ([out angle=-30].2,-.3) .. (-45:.8cm);
        \draw ([out angle=-150].2,-.3) ..([in angle=-70]135:.8cm);
 	\draw
              ([out angle=150].2,.3) .. ([in angle=70]-135:.8cm);
\end{tikzpicture}}

\newfig{\drawsyminvcrossX}{\drawsyminvcrossXbox}{\begin{tikzpicture}[baseline=-0.5ex,scale=0.8,rotate=-90]
 	\draw ([out angle=30].2,.3) .. (45:.8cm);
	\draw (.2,.3) -- (.2,-.3);
 	\draw ([out angle=-30].2,-.3) .. (-45:.8cm);
        \draw ([out angle=-150].2,-.3) ..([in angle=-70]135:.8cm);
 	\draw
              ([out angle=150].2,.3) .. ([in angle=70]-135:.8cm);
\end{tikzpicture}}

\newfig{\twist}{\twistbox}{%
  \begin{tikzpicture}[baseline=-0.5ex,scale=0.8]
    \draw[line width=1mm,white,double=black]
       ([out angle=-15]135:.8cm)..([blank=soft]-60:.4cm)..(-120:.4cm)..([in angle=-165]45:.8cm);
    \draw[line width=1mm,white,double=black,use previous Hobby path={invert soft blanks,disjoint}];
  \end{tikzpicture}}

\newfig{\symtwist}{\symtwistbox}{%
  \begin{tikzpicture}[baseline=-0.5ex,scale=0.8]
    \draw
       ([out angle=-15]135:.8cm)..([blank=soft]-60:.4cm)..(-120:.4cm)..([in angle=-165]45:.8cm);
   \draw[use previous Hobby path={invert soft blanks,disjoint}];
  \end{tikzpicture}}

\newfig{\twistvertex}{\twistvertexbox}{%
\begin{tikzpicture}[baseline=-0.5ex,scale=0.8]
  \draw (0,-0.4)--(0,-0.8);
  \draw ([out angle=-170]30:0.8cm)..([in angle=160](0,-0.4);
  \draw[line width=1mm,white,double=black]
        ([out angle=-10]150:0.8cm)..([in angle=20](0,-0.4);
\end{tikzpicture}}

\newfig{\symtwistvertex}{\symtwistvertexbox}{%
\begin{tikzpicture}[baseline=-0.5ex,scale=0.8]
  \draw (0,-0.4)--(0,-0.8);
  \draw ([out angle=-170]30:0.8cm)..([in angle=160](0,-0.4);
  \draw
        ([out angle=-10]150:0.8cm)..([in angle=20](0,-0.4);
\end{tikzpicture}}

\newfig{\loopvertex}{\loopvertexbox}{%
\begin{tikzpicture}[baseline=-0.5ex,scale=0.8]
  \draw (0,-0.5)--(0,-0.8);
  \draw ([out angle=150]0,-0.5)..(-0.02,0.6)..(0.02,0.6)..([in angle=30]0,-0.5);
\end{tikzpicture}}

\newfig{\idtangle}{\idtanglebox}{%
\begin{tikzpicture}[baseline=-0.5ex,scale=0.8]
\draw (0,0) node {$D$};
\draw[densely dashed] (0,0) circle (.5cm);
\draw (45:.5cm) -- (45:0.8cm);
\draw (135:.5cm) -- (135:0.8cm);
\draw (-45:.5cm) -- (-45:0.8cm);
\draw (-135:.5cm) -- (-135:0.8cm);
\end{tikzpicture}}
    
\newfig{\Rcycle}{\Rcyclebox}{%
\begin{tikzpicture}[baseline=-0.5ex,scale=0.8]
\draw (0,0) node {$D$};
\draw[densely dashed] (0,0) circle (.4cm);
\draw (45:.4cm) -- (45:1.1cm);
\draw ([out angle=-70]-45:.4cm) .. (-90:.7cm) .. ([in angle=0]-135:1.1cm);
\draw ([out angle=-160]-135:.4cm) .. (180:.7cm) .. ([in angle=-90]135:1.1cm);
\draw[line width=0.8mm,white,double=black]
      ([out angle=90]135:.4cm) .. (45:.7cm) .. ([in angle=90]-45:1.1cm);
\end{tikzpicture}}

\newfig{\ladder}{\ladderbox}{%
\begin{tikzpicture}[baseline=-0.5ex,scale=0.8]
\draw (0,0) node {$D$};
\draw[densely dashed] (0,0) circle (.5cm);
\draw (135:.5cm) -- (135:0.9cm);
\draw (-135:.5cm) -- (-135:0.9cm);
\draw ([out angle=45]45:.5cm) .. ([in angle=150]0.8,0.45);
\draw ([out angle=-45]-45:.5cm) .. ([in angle=-150]0.8,-0.45);
\draw (0.8,0.45)--(0.8,-0.45);
\draw ([out angle=30]0.8,0.45) .. ([in angle=-135]1.2,0.636);
\draw ([out angle=-30]0.8,-0.45) .. ([in angle=135]1.2,-0.636);
\end{tikzpicture}}

\newfig{\sqsq}{\sqsqbox}{%
\begin{tikzpicture}[scale=.15]
\draw (0,0) -- (2,0) -- (2,1) -- (0,1) -- cycle (1,0) -- (1,1);
\end{tikzpicture}}

\newfig{\overviolin}{\overviolinbox}{%
\begin{tikzpicture}[baseline=0]
	\coordinate (E1) at (45:.8cm);
	\coordinate (E2) at (225:.8cm);
	\coordinate (E3) at (270:.8cm);
	\coordinate (P1) at (0,0);
	\coordinate (P2) at (120:.6cm);
	\coordinate (P3) at (.4cm,.2cm);
	\coordinate (C1) at (100:.9cm);
	\coordinate (C2) at (135:.5cm);
	\coordinate (C3) at (.8,.1);
	\draw (E1) -- (P1) -- (E3);
	\draw (P1) .. controls  (C1) .. (P2);
	\draw[white, line width=4pt] (P2) .. controls  (C2) .. (P3);	
	\draw[white, line width=4pt] (P3) .. controls (C3) .. (E2);
	\draw (P2) .. controls  (C2) .. (P3);	
	\draw (P3) .. controls (C3) .. (E2);
\end{tikzpicture}}

\newfig{\rotatedtrivalent}{\rotatedtrivalentbox}{%
\begin{tikzpicture}[baseline=0]
	\coordinate (E1) at (45:.8cm);
	\coordinate (E2) at (225:.8cm);
	\coordinate (E3) at (270:.8cm);
	\coordinate (P0) at (0,0);
	\coordinate (C1) at (-.1,.7);
	\coordinate (C2) at (-.3,-.2);
	\coordinate (C3) at (.2,-.2);
	\draw (P0) .. controls (C1) .. (E2);
	\draw (P0) .. controls (C2) .. (E3);
	\draw (P0) .. controls (C3) .. (E1);
\end{tikzpicture}}

\newcommand{\dogbone}{%
\begin{tikzpicture}[baseline=0cm]
	\node (a) at (-.5,0) {$\bullet$};
	\node (b) at (.5,0) {$\bullet$};
	\draw (a.center) to (b.center);
\end{tikzpicture}}

\newcommand{\counitunit}{%
\begin{tikzpicture}[baseline=-0.5ex]
    \node (a) at (0,.15) {$\bullet$};
    \node (b) at (0,-.15) {$\bullet$};
    \draw (a.center) -- ++(0,.45);
    \draw (b.center) -- ++(0,-.45);
\end{tikzpicture}%
}

\newcommand{\onevalent}{%
\begin{tikzpicture}[baseline=0cm]
	\node (a) at (0,-.5) {};
	\node (b) at (0,.5) {$\bullet$};
	\draw (a.center) to (b.center);
\end{tikzpicture}}

\newcommand{\leftunitvertex}{%
\begin{tikzpicture}[baseline=0cm]
	\node (a) at (-.5,.5) {$\bullet$};
	\node (b) at (.5,.5) {};
	\node (c) at (0,-.5) {};
	\node (d) at (0,0) {};
	\draw (a.center) to (d.center);
	\draw (b.center) to (d.center);
	\draw (c.center) to (d.center);
\end{tikzpicture}}

\newcommand{\rightunitvertex}{%
\begin{tikzpicture}[baseline=0cm]
	\node (a) at (-.5,.5) {};
	\node (b) at (.5,.5) {$\bullet$};
	\node (c) at (0,-.5) {};
	\node (d) at (0,0) {};
	\draw (a.center) to (d.center);
	\draw (b.center) to (d.center);
	\draw (c.center) to (d.center);
\end{tikzpicture}}

\tikzset{
  Pcoupon/.style={
    rectangle,
    draw,
    fill=white,
    inner sep=0pt,
    outer sep=0pt,
    minimum width=4.5mm,
    minimum height=3mm,
    font=\scriptsize
  },
  unitvertex/.style={
    circle,
    fill,
    inner sep=0pt,
    minimum size=1.4mm
  }
}

\newcommand{\PBigon}[2]{%
\begin{tikzpicture}[
    baseline={(0,.60)},
    line width=.45pt,
    x=.9cm,
    y=.9cm
  ]
  \draw (0,-.55) -- (0,-.47);
  \draw (0,-.13) -- (0,0);
  \node[Pcoupon] at (0,-.30) {$P$};

  \draw (0,1.20) -- (0,1.33);
  \draw (0,1.67) -- (0,1.75);
  \node[Pcoupon] at (0,1.50) {$P$};

  \ifcase#1
    \draw (0,0)
      .. controls (-.36,.23) and (-.36,.97) ..
      (0,1.20);
  \or
    \draw (0,0)
      .. controls (-.29,.17) and (-.35,.34) ..
      (-.34,.43);
    \draw (-.34,.77)
      .. controls (-.35,.86) and (-.29,1.03) ..
      (0,1.20);
    \node[Pcoupon] at (-.34,.60) {$P$};
  \or
    \draw (0,0)
      .. controls (-.29,.17) and (-.34,.36) ..
      (-.33,.48);
    \node[unitvertex] at (-.33,.48) {};
    \node[unitvertex] at (-.33,.72) {};
    \draw (-.33,.72)
      .. controls (-.34,.84) and (-.29,1.03) ..
      (0,1.20);
  \fi

  \ifcase#2
    \draw (0,0)
      .. controls (.36,.23) and (.36,.97) ..
      (0,1.20);
  \or
    \draw (0,0)
      .. controls (.29,.17) and (.35,.34) ..
      (.34,.43);
    \draw (.34,.77)
      .. controls (.35,.86) and (.29,1.03) ..
      (0,1.20);
    \node[Pcoupon] at (.34,.60) {$P$};
  \or
    \draw (0,0)
      .. controls (.29,.17) and (.34,.36) ..
      (.33,.48);
    \node[unitvertex] at (.33,.48) {};
    \node[unitvertex] at (.33,.72) {};
    \draw (.33,.72)
      .. controls (.34,.84) and (.29,1.03) ..
      (0,1.20);
  \fi
\end{tikzpicture}%
}

\newcommand{\PIdentity}{%
\begin{tikzpicture}[
    baseline={(0,.60)},
    line width=.45pt,
    x=.9cm,
    y=.9cm
  ]
  \draw (0,-.10) -- (0,.43);
  \draw (0,.77) -- (0,1.30);
  \node[Pcoupon] at (0,.60) {$P$};
\end{tikzpicture}%
}

\newcommand{\twistone}{%
\begin{tikzpicture}[
    baseline=(current bounding box.center),
    scale=1.5,
    every path/.style={
        draw,
        line width=.45pt,
        line cap=round,
        line join=round
    }
]
    \draw
        (0,0)
        -- (.34,.34)
        .. controls (.43,.43) and (.46,.56) ..
        (.375,.625);

    \draw[
        preaction={draw=white,line width=2.2pt}
    ]
        (.5,0)
        -- (.16,.34)
        .. controls (.07,.43) and (.04,.56) ..
        (.125,.625);

    \draw
        (.125,.625)
        .. controls (.07,.66) and (.025,.71) ..
        (0,.75);

    \draw
        (.375,.625)
        .. controls (.43,.66) and (.475,.71) ..
        (.5,.75);

    \draw
        (.125,.625)
        -- (.375,.625);
\end{tikzpicture}%
}

\newcommand{\twisttwo}{%
\begin{tikzpicture}[
    baseline=(current bounding box.center),
    scale=1.5,
    every path/.style={
        draw,
        line width=.45pt,
        line cap=round,
        line join=round
    }
]
    \draw
        (.25,.25)
        -- (-.14,-.14)
        .. controls (-.23,-.23) and (-.32,-.21) ..
        (-.375,-.125);

    \draw[
        preaction={draw=white,line width=2.2pt}
    ]
        (.25,-.25)
        -- (-.14,.14)
        .. controls (-.23,.23) and (-.32,.21) ..
        (-.375,.125);

    \draw
        (-.375,.125)
        .. controls (-.42,.18) and (-.47,.23) ..
        (-.5,.25);

    \draw
        (-.375,-.125)
        .. controls (-.42,-.18) and (-.47,-.23) ..
        (-.5,-.25);

    \draw
        (-.375,-.125)
        -- (-.375,.125);
\end{tikzpicture}%
}

%% file: deformations.bib
@misc{stacks-project,
  author       = {The {Stacks project authors}},
  title        = {The Stacks project},
  howpublished = {\url{https://stacks.math.columbia.edu}},
  year         = {2026},
}

@unpublished{EtingofSnowdenDeformations,
  author = {Etingof, Pavel and Snowden, Andrew},
  title  = {Deformations of symmetric tensor categories},
  note   = {Work in progress},
}

@misc{2510.09922,
      title={Reconstruction of tensor categories of type ${G}_2$}, 
      author={Lilit Martirosyan and Hans Wenzl},
      year={2026},
      eprint={2510.09922},
      archivePrefix={arXiv},
      primaryClass={math.QA},
      url={https://arxiv.org/abs/2510.09922}, 
}

@incollection {MR1664995,
    AUTHOR = {Yetter, David N.},
     TITLE = {Braided deformations of monoidal categories and {V}assiliev
              invariants},
 BOOKTITLE = {Higher category theory ({E}vanston, {IL}, 1997)},
    SERIES = {Contemp. Math.},
    VOLUME = {230},
     PAGES = {117--134},
 PUBLISHER = {Amer. Math. Soc., Providence, RI},
      YEAR = {1998},
      ISBN = {0-8218-1056-1},
   MRCLASS = {18D10 (57M25)},
  MRNUMBER = {1664995},
       DOI = {10.1090/conm/230/03341},
       URL = {https://doi-org.proxyiub.uits.iu.edu/10.1090/conm/230/03341},
}

@article {MR4756485,
    AUTHOR = {Orellana, Rosa and Wallace, Nancy and Zabrocki, Mike},
     TITLE = {Representations of the quasi-partition algebras},
   JOURNAL = {J. Algebra},
  FJOURNAL = {Journal of Algebra},
    VOLUME = {655},
      YEAR = {2024},
     PAGES = {758--793},
      ISSN = {0021-8693,1090-266X},
   MRCLASS = {20C30 (05E16 20C05)},
  MRNUMBER = {4756485},
MRREVIEWER = {Annamalai\ Tamilselvi},
       DOI = {10.1016/j.jalgebra.2024.01.025},
       URL = {https://doi-org.proxyiub.uits.iu.edu/10.1016/j.jalgebra.2024.01.025},
}

@book{EGHLSVY11,
  author    = {Pavel Etingof and Oleg Golberg and Sebastian Hensel
               and Tiankai Liu and Alex Schwendner and Dmitry Vaintrob
               and Elena Yudovina},
  title     = {Introduction to Representation Theory},
  series    = {Student Mathematical Library},
  volume    = {59},
  publisher = {American Mathematical Society},
  address   = {Providence, RI},
  year      = {2011},
  isbn      = {978-0-8218-5351-1}
}

@article {MR1641842,
    AUTHOR = {Crane, L. and Yetter, D. N.},
     TITLE = {Deformations of (bi)tensor categories},
   JOURNAL = {Cahiers Topologie G\'eom. Diff\'erentielle Cat\'eg.},
  FJOURNAL = {Cahiers de Topologie et G\'eom\'etrie Diff\'erentielle
              Cat\'egoriques},
    VOLUME = {39},
      YEAR = {1998},
    NUMBER = {3},
     PAGES = {163--180},
      ISSN = {0008-0004},
   MRCLASS = {18D10 (17B37)},
  MRNUMBER = {1641842},
MRREVIEWER = {Sorin\ D\u asc\u alescu},
}

@article {MR2308953,
    AUTHOR = {Sikora, Adam S. and Westbury, Bruce W.},
     TITLE = {Confluence theory for graphs},
   JOURNAL = {Algebr. Geom. Topol.},
  FJOURNAL = {Algebraic \& Geometric Topology},
    VOLUME = {7},
      YEAR = {2007},
     PAGES = {439--478},
      ISSN = {1472-2747,1472-2739},
   MRCLASS = {57M15 (05C10 05C30 17B20)},
  MRNUMBER = {2308953},
MRREVIEWER = {S.\ V.\ Duzhin},
       DOI = {10.2140/agt.2007.7.439},
       URL = {https://doi-org.proxyiub.uits.iu.edu/10.2140/agt.2007.7.439},
}

@unpublished{SpencerThesis,
  author = {Benjamin Spencer},
  title  = {Rep({S}t)-like categories},
  note   = {Ph.D. thesis, Indiana University, in preparation},
  year   = {2026},
}

@article {MR2143201,
    AUTHOR = {Halverson, Tom and Ram, Arun},
     TITLE = {Partition algebras},
   JOURNAL = {European J. Combin.},
  FJOURNAL = {European Journal of Combinatorics},
    VOLUME = {26},
      YEAR = {2005},
    NUMBER = {6},
     PAGES = {869--921},
      ISSN = {0195-6698,1095-9971},
   MRCLASS = {05E10 (20C30 20G05)},
  MRNUMBER = {2143201},
MRREVIEWER = {Karin\ Erdmann},
       DOI = {10.1016/j.ejc.2004.06.005},
       URL = {https://doi-org.proxyiub.uits.iu.edu/10.1016/j.ejc.2004.06.005},
}

@misc{2204.04526,
      title={Oligomorphic groups and tensor categories}, 
      author={Nate Harman and Andrew Snowden},
      year={2024},
      eprint={2204.04526},
      archivePrefix={arXiv},
      primaryClass={math.RT},
      url={https://arxiv.org/abs/2204.04526}, 
}

@article {MR2559686,
    AUTHOR = {Morrison, Kim and Peters, Emily and Snyder, Noah},
     TITLE = {Skein theory for the {$D_{2n}$} planar algebras},
   JOURNAL = {J. Pure Appl. Algebra},
  FJOURNAL = {Journal of Pure and Applied Algebra},
    VOLUME = {214},
      YEAR = {2010},
    NUMBER = {2},
     PAGES = {117--139},
      ISSN = {0022-4049,1873-1376},
   MRCLASS = {46L37 (18D10 57M15)},
  MRNUMBER = {2559686},
MRREVIEWER = {Yasuyuki\ Kawahigashi},
       DOI = {10.1016/j.jpaa.2009.04.010},
       URL = {https://doi-org.proxyiub.uits.iu.edu/10.1016/j.jpaa.2009.04.010},
}

@article {MR5101335,
    AUTHOR = {Morrison, Kim and Snyder, Noah and Thurston, Dylan},
     TITLE = {Towards the {Q}uantum {E}xceptional {S}eries},
   JOURNAL = {Mem. Amer. Math. Soc.},
  FJOURNAL = {Memoirs of the American Mathematical Society},
    VOLUME = {321},
      YEAR = {2026},
    NUMBER = {1637},
     PAGES = {vii+109},
      ISSN = {0065-9266,1947-6221},
      ISBN = {978-1-4704-8148-3; 978-1-4704-8709-6},
   MRCLASS = {57K14 (17B25 18M15 18M30 20)},
  MRNUMBER = {5101335},
       DOI = {10.1090/memo/1637},
       URL = {https://doi-org.proxyiub.uits.iu.edu/10.1090/memo/1637},
}

@unpublished{OstrikSnyderG2,
  author = {Ostrik, Victor and Snyder, Noah},
  title  = {Tilting Modules for Quantum {$G_2$} at Roots of Unity and {K}uperberg's Spider},
  note   = {Work in progress}
}

@misc{1507.06030,
      title={Yang-Baxter relation planar algebras}, 
      author={Zhengwei Liu},
      year={2016},
      eprint={1507.06030},
      archivePrefix={arXiv},
      primaryClass={math.OA},
      url={https://arxiv.org/abs/1507.06030}, 
}

@incollection {MR1237835,
    AUTHOR = {Kazhdan, David and Wenzl, Hans},
     TITLE = {Reconstructing monoidal categories},
 BOOKTITLE = {I. {M}. {G}elfand {S}eminar},
    SERIES = {Adv. Soviet Math.},
    VOLUME = {16, Part 2},
     PAGES = {111--136},
 PUBLISHER = {Amer. Math. Soc., Providence, RI},
      YEAR = {1993},
      ISBN = {0-8218-4119-X},
   MRCLASS = {18D10 (17B37)},
  MRNUMBER = {1237835},
MRREVIEWER = {Graham\ J.\ Ellis},
}

@ARTICLE{MR2737787,
  AUTHOR = {Comes, Jonathan and Ostrik, Victor},
  YEAR = {2011},
  DOI = {10.1016/j.aim.2010.08.010},
  ISSN = {0001-8708},
  JOURNAL = {Adv. Math.},
  NUMBER = {2},
  PAGES = {1331--1377},
  TITLE = {On blocks of {D}eligne's category\ {$\underline{\rm Re}{\rm p}(S_t)$}},
  VOLUME = {226},
}

@ARTICLE{MR3177889,
  AUTHOR = {Daugherty, Zajj and Orellana, Rosa},
  YEAR = {2014},
  DOI = {10.1016/j.jalgebra.2013.11.028},
  ISSN = {0021-8693},
  JOURNAL = {J. Algebra},
  ARXIV = {1212.2596},
  PAGES = {124--151},
  TITLE = {The quasi-partition algebra},
  VOLUME = {404},
}

@INCOLLECTION{MR2348906,
  AUTHOR = {Deligne, Pierre},
  address = {Mumbai},
  BOOKTITLE = {Algebraic groups and homogeneous spaces},
  YEAR = {2007},
  NOTE = {available at \url{http://www.math.ias.edu/files/deligne/Symetrique.pdf}},
  PAGES = {209--273},
  SERIES = {Tata Inst. Fund. Res. Stud. Math.},
  volume = 	 {19},
  TITLE = {La catégorie des représentations du groupe symétrique {$S_t$}, lorsque {$t$} n'est pas un entier naturel},
}

@MISC{1601.03426,
  AUTHOR = {Harman, Nate},
  YEAR = {2016},
  ARXIV = {1601.03426},
  TITLE = {Deligne categories as limits in rank and characteristic},
}

@MISC{2007.11640,
  AUTHOR = {Khovanov, Mikhail and Sazdanovic, Radmila},
  YEAR = {2020},
  ARXIV = {2007.11640},
  TITLE = {Bilinear pairings on two-dimensional cobordisms and generalizations of the {D}eligne category},
}

@ARTICLE{MR1265145,
  AUTHOR = {Kuperberg, Greg},
  YEAR = {1994},
  ISSN = {0129-167X},
  JOURNAL = {Internat. J. Math.},
  ARXIV = {math.QA/9201302},
  DOI = {10.1142/S0129167X94000048},
  NUMBER = {1},
  PAGES = {61--85},
  TITLE = {The quantum {$G\sb 2$} link invariant},
  VOLUME = {5},
}

@ARTICLE{MR1403861,
  AUTHOR = {Kuperberg, Greg},
  YEAR = {1996},
  ISSN = {0010-3616},
  JOURNAL = {Comm. Math. Phys.},
  ARXIV = {q-alg/9712003},
  NUMBER = {1},
  PAGES = {109--151},
  TITLE = {Spiders for rank {$2$} {L}ie algebras},
  VOLUME = {180},
       URL = {http://projecteuclid.org/euclid.cmp/1104287237}
}

@ARTICLE{MR2783128,
  AUTHOR = {Morrison, Kim and Peters, Emily and Snyder, Noah},
  YEAR = {2011},
  DOI = {10.4171/QT/16},
  ISSN = {1663-487X},
  JOURNAL = {Quantum Topol.},
  ARXIV = {1003.0022},
  NUMBER = {2},
  PAGES = {101--156},
  TITLE = {Knot polynomial identities and quantum group coincidences},
  VOLUME = {2},
}

@ARTICLE{MR3624901,
  AUTHOR = {Morrison, Kim and Peters, Emily and Snyder, Noah},
  YEAR = {2017},
  DOI = {10.1007/s00029-016-0240-3},
  ISSN = {1022-1824},
  JOURNAL = {Selecta Math. (N.S.)},
  ARXIV = {1501.06869},
  NUMBER = {2},
  PAGES = {817--868},
  TITLE = {Categories generated by a trivalent vertex},
  VOLUME = {23},
}

@INCOLLECTION{MR2767048,
  AUTHOR = {Selinger, P.},
  PUBLISHER = {Springer, Heidelberg},
  BOOKTITLE = {New structures for physics},
  YEAR = {2011},
  DOI = {10.1007/978-3-642-12821-9\_4},
  ARXIV = {0908.3347},
  PAGES = {289--355},
  SERIES = {Lecture Notes in Phys.},
  TITLE = {A survey of graphical languages for monoidal categories},
  VOLUME = {813},
}

@ARTICLE{MR2132671,
  AUTHOR = {Tuba, Imre and Wenzl, Hans},
  YEAR = {2005},
  ISSN = {0075-4102},
  JOURNAL = {J. Reine Angew. Math.},
  ARXIV = {math.QA/0301142},
  DOI = {10.1515/crll.2005.2005.581.31},
  PAGES = {31--69},
  TITLE = {On braided tensor categories of type {$BCD$}},
  VOLUME = {581},
}

@article {MR4450143,
    AUTHOR = {Khovanov, Mikhail and Ostrik, Victor and Kononov, Yakov},
     TITLE = {Two-dimensional topological theories, rational functions and
              their tensor envelopes},
   JOURNAL = {Selecta Math. (N.S.)},
  FJOURNAL = {Selecta Mathematica. New Series},
    VOLUME = {28},
      YEAR = {2022},
    NUMBER = {4},
     PAGES = {Paper No. 71, 68},
      ISSN = {1022-1824,1420-9020},
   MRCLASS = {18M10 (05A18 57R56 81R05)},
  MRNUMBER = {4450143},
       DOI = {10.1007/s00029-022-00785-z},
       URL = {https://doi.org/10.1007/s00029-022-00785-z},
  ARXIV = 	 {2011.14758},
}
